\documentclass[english,11pt]{article}
\usepackage[T1]{fontenc}
\usepackage{babel}
\usepackage{amsmath,amsfonts,amsthm}
\usepackage{color}
\usepackage{pdfsync}
\usepackage{graphicx}
\usepackage{url}
\usepackage[noinfoline]{imsart}

\DeclareMathSymbol{\leqslant}{\mathalpha}{AMSa}{"36} 
\DeclareMathSymbol{\geqslant}{\mathalpha}{AMSa}{"3E} 
\renewcommand{\leq}{\;\leqslant\;}                   
\renewcommand{\geq}{\;\geqslant\;}                   

\newtheorem{Th}{Theorem}
\newtheorem{Le}[Th]{Lemma}
\newtheorem{Pro}[Th]{Proposition}

\newcommand{\bbN}{{\ensuremath{\mathbb N}} }

\newcommand{\bbP}{{\ensuremath{\mathbb P}} }

\newcommand{\bbR}{{\ensuremath{\mathbb R}} }

\newcommand{\N}{\bbN}

\newcommand{\R}{\bbR}

\begin{document}
\begin{frontmatter}
\title{Chaotic extensions and the lent particle method for Brownian motion }
\date{}
\runtitle{}
\author{\fnms{Nicolas}
 \snm{BOULEAU}\corref{}\ead[label=e2]{bouleau@enpc.fr}}
\address{Ecole des Ponts, Paris Tech, Paris-Est \\
6 Avenue Blaise Pascal\\77455 Marne-La-Vallée Cedex 2- FRANCE \\ \printead{e2}}

\author{\fnms{Laurent}
 \snm{DENIS}\corref{}\ead[label=e1]{ldenis@univ-evry.fr}}
\thankstext{T1}{The work of the second author is supported by the chair \textit{risque de cr\'edit}, F\'ed\'eration bancaire Fran\c{c}aise}
\address{D\'epartement de Math\'ematiques\\ Laboratoire  Analyse et Probabilit\'es\\Universit\'e
d'Evry-Val-d'Essonne\\23 Boulevard de France\\91037 EVRY Cedex
-FRANCE\\\printead{e1}}

\runauthor{N. Bouleau and L. Denis}

\begin{abstract}
In previous works, we have developed a new Malliavin calculus on the Poisson space based on {\it the lent particle formula}. The aim of this work is to prove that, on the Wiener space for the standard Ornstein-Uhlenbeck structure, we also have such a formula which permits to calculate easily and intuitively the Malliavin derivative of a functional. Our approach uses   chaos extensions associated to stationary processes of rotations of normal martingales.
\end{abstract}


\end{frontmatter}

\section{ Introduction}

When a measurable space with a $\sigma$-finite measure $\nu$ is equipped on  $L^2(\nu)$ with a local Dirichlet form with carr\'e du champ
 $\gamma$, the associated Poisson space, i.e. the probability space of a random Poisson measure with intensity $\nu$, may itself be endowed with a local Dirichlet structure with carr\'e du champ $\Gamma$  (cf. \cite{surgailis}, \cite{wu}). If a gradient  $\flat$ has been chosen associated with the operator 
$\gamma$, a gradient $\sharp$ associated with   $\Gamma$ may be constructed on the  Poisson space (cf \cite{akr},\cite{ma-rockner2},\cite{privault},\cite{bouleau-denis}) and we have shown  \cite{bouleau-denis},\cite{bouleau-denis2}, that such a gradient is provided by the lent particle formula which amounts to add a point to the configuration, to derive with respect to this point, and then to take back the point before integrating with respect to a random Poisson measure variant of the initial one. 

On the example of a L\'evy process, in order to find the gradient of the functional  $ V=\int_0^t\varphi(Y_{u-})dY_u$, this method consists in adding a jump to the process  $Y$ at time $s$ and then deriving with respect to the size of this jump. 

If we think the Brownian motion as a L\'evy process, this addresses naturally the question of knowing whether  to obtain the Malliavin derivative of a Wiener functional we could add a jump to the Brownian path and derive with respect to the size of the jump, in other words whether we have, denoting $D_sF$ the Malliavin derivative of $F$
\begin{equation}\label{un}D_sF=\lim_{a\rightarrow0}\frac{1}{a}(F(\omega+a1_{\{.\geq s\}})-F(\omega)).\end{equation}
Formula (\ref{un}) is satisfied in the case
$F=\Phi(\int_0^1 h_1dB, \ldots, \int_0^1 h_n dB)$ with $\Phi$ regular and $h_i$ continuous. But this formula has no sense in general, since the mapping $t\mapsto1_{\{t\geq s\}}$ does not belong to the  Cameron-Martin space.

We tackle this question by means of the  {\it chaotic extension} of a Wiener functional to a normal martingale weighted combination of a Brownian motion and a Poisson process, and we show that the gradient and its domain are characterized in terms of derivative of a second order stationary process.

We show that a formula similar to  (\ref{un}) is valid and yields the gradient if  $F$ belongs to the domain of the Ornstein-Uhlenbeck Dirichlet form, but whose meaning and justification involve chaotic decompositions. This gives rise to a concrete calculus allowing   $\mathcal{C}^1$ changes of variables.\\
Let us also mention the works of B. Dupire (\cite{dupire}), R. Cont and D.A. Fournié (\cite{cont}) which use an idea somewhat similar but in a completely different mathematical approach and context.

\section{The second order stationary process of rotations of normal martingales.}\label{Section1}
Let $B$ be a standard one-dimensional Brownian motion defined on the Wiener space $\Omega_1$ under the Wiener measure $\mathbb{P}_1$.\\
In this section, we consider $\tilde{N}$ a standard compensated Poisson process independent of  $B$. We denote by $\mathbb{P}_2$ the law of the  Poisson process $N$ and $\mathbb{P}=\mathbb{P}_1\times\mathbb{P}_2$. \\
Let us point out that in the next sections, we shall replace $\tilde{N}$ by any normal martingale.
\subsection{The chaotic extension}\label{ChaoticExtension}
For real $\theta$, let us consider the normal martingale
$$X^\theta_t=B_t\cos\theta+\tilde{N}sin\theta .$$ 

If $f_n$ is a symmetric function of  $L^2(\mathbb{R}^n,\lambda_n)$, we denote $I_n(f_n)$ the Brownian multiple stochastic integral  and $I_n^\theta(f_n)$ the multiple stochastic integral with respect to  $X^\theta$. We have classically cf \cite{dellacherie}
$$\|I_n(f_n)\|^2_{L^2(\mathbb{P}_1)}=\|I^\theta_n(f_n)\|^2_{L^2(\mathbb{P})}=n!\|f_n\|^2_{L^2(\lambda_n)}.$$
It follows that if $F\in L^2(\mathbb{P}_1)$ has the expansion on the Wiener chaos $$F=\sum_{n=0}^\infty I_n(f_n)$$ the same sequence $f_n$ defines {\it a chaotic extension} of $F$ : $F^\theta=\sum_{n=0}^\infty I_n^\theta(f_n)$. \\

\noindent{\bf Remark 1.} Let us emphasize that the {\it chaotic extension}  $F\mapsto F^\theta$  is not compatible with composition of applications : $\Phi\circ\Psi^\theta\neq(\Phi\circ\Psi)^\theta$ except obvious cases as seen by taking  $\Phi(x)=x^2$, $\Psi=I_1(f)$ and $\theta=\pi/2$. Thus it is important that the sequence  $(f_n)_n$ appears in the notation : we will use the "short notation" of  \cite{dellacherie}.\\

We denote $\mathcal{P}$ (resp. $\mathcal{P}(t)$) the set of finite subsets of  $]0,\infty[$ (resp. $]0,t]$).  We write $A=\{s_1<\cdots<s_n\}$ for the current element of  $\mathcal{P}$ and $dA$ for the measure whose restriction to each simplex is the Lebesgue measure, cf \cite{dellacherie} p201 {\it et seq.}

If $F\in L^2(\mathbb{P}_1)$ expans in $F=\sum_{n=0}^\infty I_n(f_n)$ we denote $f\in L^2(\mathcal{P})$ the sequence $f=(n!f_n)_{n\in\mathbb{N}}$ and
$$\begin{array}{rl}
I^\theta(f)&=\int_\mathcal{P}f(A)dX^\theta_A\\
&\\
&=f(\emptyset)+\sum_{n>0}\int_{s_1<\cdots<s_n}f(s_1,\ldots,s_n)dX^\theta_{s_1}\cdots dX^\theta_{s_n}.
\end{array}$$ Thus we have $f(\emptyset)=\mathbb{E}[F]$, $F=I^0(f)$ and $I^{\pi/2}(f)=\int_\mathcal{P}f(A)d\tilde{N}_A=\sum_nI^{\pi/2}_n(f_n)$.

In all the paper we confond the stochastic integrals $\int H_{s-}dX^\theta_s$ and $\int H_{s}dX^\theta_s$ thanks to the fact that  $X^\theta$ is normal.
\begin{Pro} Let be $f$ and $g\in L^2(\mathcal{P})$, and $h=f+ig\in L^2_\mathbb{C}(\mathcal{P})$. The random variable  $H^\theta=\int_\mathcal{P}h(A)dX^\theta_A$ defines a second order stationary process  continuous in $L^2_\mathbb{C}(\mathbb{P})$.
\end{Pro}
\begin{proof} Let us denote similarly  $F^\theta=\int_\mathcal{P}f(A)dX^\theta_A$ and   $G^\theta=\int_\mathcal{P}g(A)dX^\theta_A$. It is enough to show that  $F^\theta$ and $G^\theta$ are measurable, second order stationary and  stationary correlated. Using the chaos expansions  $F^{\theta+\varphi}=\sum_nI_n^{\theta+\varphi}(f_n)$ and $G^{\theta}=\sum_nI_n^{\theta}(g_n)$ that comes from the fact that the bracket of the martingales $X^{\theta+\varphi}$ and $X^\theta$ is 
\begin{equation}\label{ZthetaplusphiZtheta}\langle X^{\theta+\varphi},X^\theta\rangle_t=t\cos\varphi\end{equation}
So $\mathbb{E}[I_n^{\theta+\varphi}(f_n)I_n^\theta(g_n)]=n!\langle f_n,g_n\rangle_{L^2(\lambda_n)}\cos^n\varphi$ and $\mathbb{E}[I_m^{\theta+\varphi}(f_m)I_n^\theta(g_n)]=0$ if $m$ is different of $n$.
\end{proof}

It follows that the stationary process  $H^\theta$ possesses  a spectral representation \begin{equation}\label{spectre}H^\theta=\sum_{n\in\mathbb{Z}}c_ne^{in\theta}\xi_n\end{equation} where the  $c_n$ are real  $\geq0$ and the $\xi_n$ are orthonormal in $L^2_\mathbb{C}(\mathbb{P})$. The norm $\|H^\theta\|^2$ which doesn't depend on  $\theta$ is the total mass of the spectral measure  $\sum c_n^2$. \\

The $c_n$ are linked with the norms of the components of  $H$ on the  chaos by formulas involving Bessel functions. In the case where  $H$ is an exponential vector, 
$$H=\sum_{k=0}^\infty I_k(\frac{h^{\otimes k}}{k!})=\exp[\int hdB-\frac{1}{2}\int h^2dt]$$
with $h=f+ig$ what implies  $\|H\|^2_{L^2_\mathbb{C}}=\exp\|h\|^2$, we have by (\ref{ZthetaplusphiZtheta}) the  covariance
$$\mathbb{E}[H^{\theta+\varphi}\overline{H^\theta}]=\sum_k\frac{1}{k!}\|h\|^{2k}_{L^2_\mathbb{C}(\lambda_1)}\cos^k\varphi=\exp(\|h\|^2\cos\varphi)$$
which is the Fourier transform of the spectral measure hence equal to  $\sum_nc_n^2e^{in\varphi}$. 

By the relation defining the Bessel functions $J_n$ (formula of Schl\"omilch)
$$\exp(iz\sin\varphi)=\sum_{n\in\mathbb{Z}}e^{in\varphi}J_n(z)\qquad z\in\mathbb{C},\quad z\neq0,$$ it comes $c_n^2=i^nJ_n(-i\|h\|^2)$ and for $n\geq0$ (cf \cite{saphar})
$$c^2_n=c^2_{-n}=\sum_{k=0}^\infty\frac{1}{k!(n+k)!}(\frac{\|h\|^2}{2})^{2k+n}.$$
The variables $c_n\xi_n$ may be also expressed in terms of Bessel functions using the expression of exponential vectors for $X^\theta$, cf formula (\ref{expvector}) below. \\

\subsection{Chaotic structure of $L^2(\mathbb{P})$.}

This part is independent of the rest of the paper. It is devoted to the study of chaotic representations for $X^\theta$.

Let us first remark that the above considerations  dont use the chaotic representation property for  $X^\theta$ which is false if $\sin\theta\cos\theta\neq0$, as it is well known cf for instance \cite{dermoune}. Let us denote $L^2(\mathbb{P}_{X^\theta})$ the vector space of $\sigma(X^\theta)$-measurable random variables belonging to $L^2(\mathbb{P})$. That means that, if $\sin\theta\cos\theta\neq0$,  the vector space $C(X^\theta)=\{F^\theta : F\in L^2(\mathbb{P}_1)\}$ which is closed in $L^2(\mathbb{P}_{X^\theta})$ has a non empty complement. 

If we consider the simplest example of the square of a functional of the first Brownian chaos $F=\int hdB$ with $h\in L^1\cap L^\infty$, we have $F^\theta=\int hdX^\theta$ and by Ito formula
$$(F^\theta)^2=2\int\int_0^th(s)dX^\theta_s\;h(t)dX^\theta_t+\int h^2ds+\sin^2\theta\int h^2d\tilde{N}$$ since $\tilde{N}=\sin\theta \,X^\theta+\cos\theta\, X^{\theta+\pi/2}$, we see that 
$$F^\theta=U+\sin^2\theta\cos\theta\int h^2dX^{\theta+\pi/2}$$ with $U\in C (X^\theta)$ and $\int h^2dX^{\theta+\pi/2}$ orthogonal to $C(X^\theta )$.

It follows that for $k\in L^2(\mathbb{R}_+)$, $\int k dX^{\theta+\pi/2}\in L^2(\mathbb{P}_{X^\theta})$ and this implies

\begin{Pro} Let us suppose $\sin\theta\cos\theta\neq0$. The $\sigma$-fields generated by $X^\theta$ and $X^{\theta+\pi/2}$ are the same. The spaces $L^2(\mathbb{P}_{X^\theta})$ do not depend on $\theta$ and are equal to $L^2(\mathbb{P})$.
\end{Pro} 
The intuitive meaning of this proposition is that on a sample path of $X^\theta$ it is possible to measurably detect the underlying Brownian and Poisson paths.

The multiple stochastic integrals w.r. to $X^\theta$ are not enough to fill $L^2(\mathbb{P}_{X^\theta})$. In view of the previous example, we may think to add the stochastic integrals w.r. to $X^{\theta+\pi/2}$, i.e. to add $C(X^{\theta+\pi/2})$, which is orthogonal to $C(X^\theta)$ and included in $L^2(\mathbb{P}_{X^\theta})$. But this is not sufficient, we must add also the crossed chaos in the following manner:

Let us consider the vector martingale $\mathbf{ X}^\theta=(X^\theta,X^{\theta+\pi/2})$. For $h=(h_1,h_2)\in L^2(\mathbb{R}_+,\mathbb{R}^2)$ we may consider the stochastic integral $\int h.d\mathbf{ X}^\theta$, and more generally (notation of \cite{bouleau-hirsch2} Chap II \S2) 
\begin{equation}\label{chaosnouveaux}\int_{\Delta_n}f.d^{(n)}\mathbf{ X}^\theta
\end{equation}
for $f\in L^2(\Delta_n,(\mathbb{R}^2)^{\otimes n})$. And using (\ref{ZthetaplusphiZtheta}) for $\varphi=\pi/2$ we have
$$\|\int_{\Delta_n}f.d^{(n)}\mathbf{ X}^\theta\|_{L^2(\mathbb{P})}=\| f\|_{L^2(\Delta_n,(\mathbb{R}^2)^{\otimes n})}.
$$ These stochastic integals define orthogonal sub-spaces of $L^2(\mathbb{P}_{X^\theta})=L^2(\mathbb{P})$. Now considering the exponential vector 
$$\mathcal{E}^\theta(h_1,h_2)=\sum_n\frac{1}{n!}\sum_{i_k\in\{1,2\}}\int h_{i_1}\otimes\cdots\otimes h_{i_n}dX^{\alpha_1}\cdots dX^{\alpha_n}$$
where $\alpha_k=\theta$ or $\theta+\pi/2$ according to $i_k=1$ or $2$, and putting $\mathcal{E}_t^\theta(h_1,h_2)$ for $\mathcal{E}^\theta(h_11_{[0,t]},h_21_{[0,t]})$, we see that the following SDE is satisfied
$$\mathcal{E}_t^\theta(h_1,h_2)=1+\int_0^t\mathcal{E}_{s-}^\theta(h_1,h_2)(h_1dX^\theta_s+h_2dX^{\theta+\pi/2}_s)$$ what gives
\begin{equation}\label{expvector}\mathcal{E}_t^\theta(h_1,h_2)=e^{V_t-\frac{1}{2}[V,V]^c_t}\prod_{s\leq t}(1+\Delta V_s)e^{-\Delta V_s}\end{equation} with $V_t=\int_0^th_1dX^\theta+\int_0^th_2dX^{\theta+\pi/2}$. We obtain
\begin{Pro}  \label{casreel} For any $\theta$, the stochastic integrals {\rm (\ref{chaosnouveaux})} define a complete orthogonal decomposition of $L^2(\mathbb{P})$.
\end{Pro}
\begin{proof} a) Let us suppose first $\sin\theta\cos\theta\neq0$. Starting with (\ref{expvector}) an easy computation yields that $\mathcal{E}_t^\theta(h_1,h_2)$ is equal up to a multiplicative constant to \break
$\exp[{\int_0^t(h_1\cos\theta-h_2\sin\theta)dB+\int_0^t ud\tilde{N}}]$ where we have taken $e^u-1=h_1\sin\theta+h_2\cos\theta$.

If we take a step  function $\xi\in L^2(\mathbb{R}_+)$ and choose $h_1$ and $h_2$ such that
$$\begin{array}{rl}
h_1\sin\theta+h_2\cos\theta&=e^{\xi\sin\theta}-1\\
h_1\cos\theta-h_2\sin\theta&=\xi\cos\theta
\end{array}
$$ we obtain that $\exp[\int\xi dX^\theta]$ belongs to the space generated by the chaos, and the result follows.

b) Now if $\theta=0$, $\mathbf{ X}^\theta=(B,\tilde{N})$. The above argument is still valid and 
$$\mathcal{E}_t^0(h_1,h_2)=\exp[\int_0^th_1dB-\frac{1}{2}\int_0^th_1^2ds+\int_0^tu_2d\tilde{N}+\int_0^tu_2ds]
$$ with $h_2=e^{u_2}-1$. That gives easily the result and the same in the other cases where $\sin\theta\cos\theta\neq0$.
\end{proof}
In other words $L^2(\mathbb{P})$ is isomorphic to the symmetric Fock space $Fock(L^2(\mathbb{R}_+,\mathbb{R}^2))$. This implies the predictable representation property with respect to $\mathbf{X}^\theta$.\\

\section{Derivative in  $\theta$ and gradient of Malliavin.}

We come back to the setting of subsection \ref{ChaoticExtension} with stochastic integrals with respect to the real process $X^\theta$. We want to study the behavior near $\theta=0$ using the fact that $X^0=B$.

But since we deal no more with chaotic representation we may replace $\tilde{N}$ by any normal martingale $M$ independent of $B$ (for instance a centered normalized L\'evy process) define under a probability that we still denote $\mathbb{P}_2$ and as in subsection \ref{ChaoticExtension}, $\mathbb{P}=\mathbb{P}_1\times\mathbb{P}_2$.  Let us put
$$Y^\theta_t=B_t\cos\theta+M_t\sin\theta$$
and consider the chaotic extensions $F\mapsto F^\theta$ with respect to $Y^\theta$ i.e. if 
$F=\sum_{n=0}^\infty I_n(f_n),$ then   $F^\theta=\sum_{n=0}^\infty I_n^\theta(f_n),$
where from now on, $I_n^\theta $ denotes  the multiple stochastic integral with respect to  $Y^\theta$.

To see the connection with the Brownian chaos expansion, let us remark that -- as in the preceding part -- for any $\theta\in\R$, the pair  $(Y_1,Y_2)=(Y^\theta,Y^{\theta+\pi/2})$ is a vector normal martingale in the sense of  \cite{dellacherie} i.e.
\begin{equation}\label{normalite}\langle Y_i,Y_j\rangle_t=\delta_{ij}t\end{equation} this allows to prove the following property 
\begin{Pro}\label{finiteexpansion}{\it For any $F$ with finite Brownian chaos expansion, $F=\sum_{k=0}^nI_n(f_n)$, $f_n$ symmetric, 
$$
\mathbb{E}[(\frac{d}{d\theta}F^\theta)^2]= \sum_{k=0}^n n!n\|f_n\|^2_{L^2}$$
}\end{Pro}

\begin{proof}
Our notation is $F^\theta=\sum_{k=0}^nI_k^\theta(f_k)$ and $F^0=F=\sum_{k=0}^nI_k(f_k)$.

Let us consider first  $I_n^\theta(f_n)$ in the case of an elementary tensor $f_n=g_1\otimes\cdots\otimes g_n$. We can write this multiple integral (with the notation of Bouleau-Hirsch  \cite{bouleau-hirsch2} p79)
$$I_n^\theta(f_n)=n!\int_{\Delta_n}f_n\,d^{(n)}Y^\theta=n!\int_0^t[\int_{\Delta_{n-1}(s)}g_1\otimes\cdots\otimes g_{n-1}d^{(n-1)}Y^\theta]g_n(s)dY^\theta_s$$
so that
$$\begin{array}{rl}
\frac{d}{d\theta}I_n^\theta(f_n)=&n!\int_0^t\frac{d}{d\theta}[\int_{\Delta_{n-1}(s)}g_1\otimes\cdots\otimes g_{n-1}d^{(n-1)}Y^\theta]g_n(s)dY^\theta_s\\
&\\
&+n!\int_0^t[\int_{\Delta_{n-1}(s)}g_1\otimes\cdots\otimes g_{n-1}d^{(n-1)}Y^\theta]g_n(s)dY^{\theta+\pi/2}_s
\end{array}
$$ hence by  (\ref{normalite})
$$\begin{array}{rl}
\mathbb{E}[(\frac{d}{d\theta}I_n^\theta(f_n))^2]=&(n!)^2\int_0^t\mathbb{E}[(\frac{d}{d\theta}[\int_{\Delta_{n-1}(s)}g_1\otimes\cdots\otimes g_{n-1}d^{(n-1)}Y^\theta])^2]g_n^2(s)ds\\
&\\
&+(n!)^2\int_0^t\mathbb{E}[(\int_{\Delta_{n-1}(s)}g_1\otimes\cdots\otimes g_{n-1}d^{(n-1)}Y^\theta)^2]g_n^2(s)ds
\end{array}
$$ what gives, running the induction down, 
$$\mathbb{E}[(\frac{d}{d\theta}I_n^\theta(f_n))^2]=n(n!)^2\int_{\Delta_{n}}(g_1\otimes\cdots\otimes g_{n-1}\otimes g_n)^2d\lambda_n=n.n!\|f_n\|^2.
$$
This extends to general tensors and similarly we can show that if  $k\neq\ell$
$$\mathbb{E}[(\frac{d}{d\theta}I_k^\theta(f_k))(\frac{d}{d\theta}I_\ell^\theta(f_\ell))]=0$$ what yields the  proposition.\end{proof}

We denote by $\mathbb{D}$, the domain of the Ornstein-Uhlenbeck form. We recall that an element $F=\sum_{n=0}^{+\infty} I_n (f_n )\in L^2 (\mathbb{P}_1)$ belongs to $\mathbb{D}$ iff
$$\sum_{n=0}^{+\infty} nn!\parallel f_n\parallel^2_{L^2}<+\infty .$$
Let us take now an  $F\in\mathbb{D}$,  the random variables  $\sum_{k=0}^nI_k^\theta(f_k)$  converge in $L^2(\mathbb{P})$ to $F^\theta$    uniformly in  $\theta$. Their derivatives  -- because $F\in\mathbb{D}$ -- form a Cauchy sequence and converge also uniformly in $\theta$. This implies that $F^\theta$ is differentiable and that the derivatives of the $\sum_{k=0}^nI_k^\theta(f_k)$ converge to the derivative of   $F^\theta$. So we have
\begin{Pro}\label{derivee1}{\it $\forall F\in\mathbb{D}$ the process $\theta\mapsto F^\theta$ is  differentiable in $L^2(\mathbb{P})$ and
$$\forall\theta\qquad\mathbb{E}[(\frac{d}{d\theta}F^\theta)^2]=\sum_{n\in\mathbb{Z}}n^2c_n^2=\mathbb{E}\Gamma[F]=2\mathcal{E}[F]$$ where $\mathcal{E}$ is the Ornstein-Uhlenbeck form and $\Gamma$ the associated  carr\'e du champ operator.
}\end{Pro}
{ We also have the converse property:
\begin{Pro} Let $F\in L^2 (\mathbb{P}_1)$. If the map $\theta \mapsto F^\theta$ is differentiable in $L^2 (\mathbb{P})$ at a certain point  $\theta_0 \in \mathbb{R}$ then $F$ belongs to $\mathbb{D}$. 
\end{Pro}
\begin{proof}
 We write $F=\sum_n I_n (f_n)$ and consider a sequence $(\theta_k )_{k\geq 1}$ which converges to $\theta_0$ and such that $\theta_k \neq \theta_0$, for all $k\geq 1$. As 
 $$\lim_{k\rightarrow +\infty}\displaystyle\frac{F^{\theta_k}-F^{\theta_0}}{\theta_k -\theta_0} $$
 exists in $L^2 (\mathbb{P})$ we deduce that there exists a constant $C>0$ such that 
 $$\forall k\geq 1 ,\ \left\| \displaystyle\frac{F^{\theta_k}-F^{\theta_0}}{\theta_k -\theta_0} \right\|^2_{L^2 (\mathbb{P})}=\sum_n  \left\| \displaystyle\frac{I_n^{\theta_k}(f_n)-I_n^{\theta_0}(f_n)}{\theta_k -\theta_0} \right\|^2_{L^2 (\mathbb{P})}\leq C.$$
By the Fatou's Lemma and the previous Proposition, we get 
  $$\sum_n \lim_{k\rightarrow +\infty} \left\| \displaystyle\frac{I_n^{\theta_k}(f_n)-I_n^{\theta_0}(f_n)}{\theta_k -\theta_0} \right\|^2_{L^2 (\mathbb{P})}=\sum_n n! n \| f_n \|^2_{L^2}\leq C,$$
which yields the result.
\end{proof}}
This provides the following result :
\begin{Pro}\label{derivee2} {\it For all $F\in\mathbb{D}$ with chaotic representation $F=\int_\mathcal{P}f(A)dB_A$
$$\frac{dF^\theta}{d\theta}|_{\theta=0}=\frac{d}{d\theta}\int_\mathcal{P}f(A)d(B\cos\theta+M\sin\theta)_A\;|_{\theta=0}=\int D_sF\;dM_s$$ the righthand term is a gradient for the Ornstein-Uhlenbeck form that we may denote  $F^\sharp$, so we have 
 in the sense of  $L^2(\mathbb{P})=L^2(\mathbb{P}_1\times\mathbb{P}_2)$
$$F^\sharp=\lim_{\theta\rightarrow0}\frac{1}{\theta}(F^\theta-F).$$
}\end{Pro}
\begin{proof}
\noindent Let be $F\in\mathbb{D}$. 
We assess the  distance between $\frac{1}{\theta}({F}^\theta-{F})$ and $\int D_sF\;dM_s$ by steps :
distance between $\frac{1}{\theta}({F}^\theta-{F})$ and $\frac{1}{\theta}({F}^\theta_n-{F}_n)$ ; between $\frac{1}{\theta}({F}^\theta_n-{F}_n)$ and $\int D_sF_n\;dM_s$ ; 
then between $\int D_sF_n\;dM_s$ and $\int D_sF\;dM_s$.
 
By the preceding propositions
 $$\|\frac{1}{\theta}({F}^\theta-{F})-\frac{1}{\theta}({F}^\theta_n-{F}_n)\|^2\leq \|{F}-{F}_n\|_\mathbb{D}^2$$
We may choose  $n$ so that the first one and the third one be both small independently of  $\theta$. And $n$ being fixed we do $\theta\rightarrow 0$ in the second one and apply the argument of the proof of Proposition \ref{finiteexpansion}.\end{proof}

The classical integration by part formula, i.e. the property that the divergence operator, dual of   $D$, can be expressed by a stochasitic integral on predictable processes, is a consequence of propositions  \ref{derivee1} and \ref{derivee2} by derivation in $\theta$.

Indeed let us denote $\mathcal{A}$ the closed sub-vector space of  $L^2(\mathbb{P},L^2(\mathbb{R}_+,dt))$ generated by the processes of the form  $\int_{\Delta(t)}f_nd^{(n)}B$ with $f_n\in L^2(\mathbb{R}^n)$.
\begin{Le} For $F\in L^2(\mathbb{P}_1)$, $G\in\mathcal{A}$, we have
$$\mathbb{E}[F^\theta\int G^\theta_tdY_t^{\theta+\pi/2}]=0.$$
\end{Le}
\begin{proof}
By relation (\ref{normalite}) the property is true if  $F$ has a finite expansion on the chaos hence also if  $F\in L^2$.\end{proof} 
Let us denote now  $\mathbb{D}_\mathcal{A}$ the closed vector space generated by the processes  $\int_{\Delta(t)}f_nd^{(n)}B$ with $f_n\in L^2(\mathbb{R}^n)$ for the norm of  $\mathbb{D}$ with values in $L^2(\mathbb{R}_+)$.
\begin{Pro} \label{parpartie}Let be $F\in \mathbb{D}$ and $G\in\mathbb{D}_\mathcal{A}$, we have 
$$
\mathbb{E}[F^\theta\int G^\theta_udY^\theta_u]=\mathbb{E}[\frac{dF^\theta}{d\theta}\int G^\theta_udY_u^{\theta+\pi/2}]$$
so that taking $\theta=0$
$$\mathbb{E}[F\int G_udB_u]=\mathbb{E}[\int D_uFdM_u\int G_udM_u]=\mathbb{E}[\int D_uF\,G_udu].$$
\end{Pro}
\begin{proof} We differentiate  $F^\theta\int G^\theta_tdY_t^{\theta+\pi/2}$ taking in account the lemma and the fact that  $Y^{\theta+\pi}=-Y^\theta$.
\end{proof}

\noindent{\bf Remark 2}. Taking anew $\tilde{N}$ for $M$, we may apply the previous reasoning  at the point $\theta=\pi/2$.  Denoting $D^{(N)}$ the operator of Nualart-Viv\`es \cite{nualart-vives}  which acts on the Poisson chaos as $D$ acts on the Brownian ones, Proposition \ref{derivee2} says that for $(f_n)$ such that $\sum n!n\|f_n\|^2<\infty$  the Poisson functional $F=\sum I_n^{\pi/2}(f_n)$ is such that $\frac{d}{d\theta}F^\theta |_{\theta=\pi/2}=\int D^{(N)}FdB$. 

And by Proposition \ref{parpartie} we obtain that the  finite difference operator $D^{(N)}$ of the Ornstein-Uhlenbeck structure on the Poisson space satisfies an integration by part formula (cf  \O ksendal and al \cite{oksendal} Thm 12.10) despite its non local character. \\

\noindent{\bf Remark 3}. In the case of another standard Brownian motion $\hat{B}$ for $M$, Proposition \ref{derivee2} gives exactly the derivation operator in the sense of Feyel-La Pradelle cf \cite{bouleau-hirsch2} Chap. II \S 2. 
$$\frac{d}{d\theta}F^\theta|_{\theta=0}=F^\prime=\int D_u Fd\hat{B}_{u}$$

In that case the situation is quite different from the one we had in Section \ref{Section1}. Indeed $Y^\theta=B\cos\theta+\hat{B}\sin\theta$ does satisfy the chaotic representation property, so that the space $\{F^\theta : F\in L^2(\mathbb{P}_1)\}$ is $L^2(\mathbb{P}_{Y^\theta})$. It is not possible to measurably detect the paths of $B$ and $\hat{B}$ on those of $Y^\theta$. But the concept of chaotic extension becomes simpler because it is compatible with the composition of the functions. It is valid to write in this case 
$$F^\theta=F(B\cos\theta+\hat{B}\sin\theta).$$
Indeed, it is correct for $F=\Phi(\int h_1dB,\ldots,\int h_kdB)$ with $\Phi$ a polynomial by Ito formula and induction (what was false in the case of the Poisson process), and then for general $F$ in $L^2$ by approximation. As a consequence, Proposition \ref{derivee2} gives a formula of Mehler type without integration for the gradient
\begin{equation}{\label{e7}}\forall F\in\mathbb{D}\qquad F^\prime=\int D_uFd\hat{B}_u=\frac{d}{d\theta}F(B\cos\theta+\hat{B}\sin\theta)|_{\theta=0}\end{equation}
and with integration for the carr\'e du champ
\begin{equation}\label{e8}\Gamma[F]=\hat{\mathbb{E}}[(\frac{d}{d\theta}F(B\cos\theta+\hat{B}\sin\theta))^2|_{\theta=0}]\end{equation} where $\hat{\mathbb{E}}$ acts on $\hat{B}$ as usual. By the change of variable $\cos\theta=e^{-t/2}$ this may be also written in a form similar to Mehler formula:
\begin{equation}\label{e9}F^\prime=\lim_{t\rightarrow 0}\frac{1}{\sqrt{t}}(F(B\sqrt{e^{-t}}+\sqrt{1-e^{-t}}\hat{B})-F(B))
\end{equation} what gives denoting $P_t$ the Ornstein-Uhlenbeck semi-group
\begin{equation}
\Gamma[F]=\lim_{t\rightarrow0}\frac{1}{t}(P_t(F^2)-2FP_tF+F^2)
\end{equation}
well known formula when $F$ and $F^2$ are in the domain of the generator and that we obtain here for $F\in \mathbb{D}$.\\
To our knowledge, formulae \eqref{e7}, \eqref{e8} and \eqref{e9} seem to be new under these hypotheses.\\

\section{Functional calculus of class $\mathcal{C}^1\cap Lip$.}

\begin{Pro} \label{int}Let us suppose that the process  $H_s$ be in $\mathbb{D}_\mathcal{A}$ {\rm(cf proposition \ref{parpartie})} then
$$(\int H_sdB_s)^\theta=\int (H_s)^\theta dY_s^\theta.$$
\end{Pro}
\begin{proof} The functional $F=\int H_sdB_s$ is in $\mathbb{D}$ and has a chaotic expansion  $F=\int_\mathcal{P}f(A)dB_A$. Following the short notation of \cite{dellacherie} (p203) if we put for  $E\in\mathcal{P}$
$$\dot{f}_t(E)=f(E\cup \{t\})\;\mbox{ if }\;E\subset[0,t[, \quad =0\;\mbox{ otherwise},$$
and if  $g_t=\int_\mathcal{P}\dot{f}_t(E)dB_E$, then
$$F=\int_\mathcal{P}f(A)dB_A=f(\emptyset)+\int g_tdB_t.$$ 
Hence we have $F^\theta=\int_\mathcal{P}f(A)dY^\theta_A$ and  $(g_t)^\theta=\int_\mathcal{P}\dot{f}_t(E)dY^\theta_E$ and
$$F^\theta=f(\emptyset)+\int (g_t)^\theta dY^\theta_t,$$
what proves the  proposition.\end{proof}
Let be  $F=(F_1,\ldots,F_k)\in\mathbb{D}^k$ and $\Phi\in\mathcal{C}^1\cap Lip(\mathbb{R}^k,\mathbb{R})$, where in a natural way $Lip (\mathbb{R}^k,\mathbb{R})$ denotes the set of uniformly Lipschitz real-valued functions defined on $\R^k$. It comes from Proposition \ref{derivee2} and from the functional calculus in local Dirichlet structures the following result 
\begin{Pro} \label{comp}When $\theta\rightarrow0$, we have in $L^2(\mathbb{P})$
$$\frac{1}{\theta}[(\Phi(F_1,\ldots,F_k))^\theta-\Phi(F_1^\theta,\ldots,F_k^\theta)]\rightarrow 0.$$
\end{Pro}
\begin{proof} We have indeed in the sense of  $L^2$, the function $\Phi$ being Lipschitz and $\mathcal{C}^1$
$$\begin{array}{rl}\lim\frac{1}{\theta}[\Phi(F_1^\theta,\ldots,F_k^\theta)-\Phi(F_1,\ldots,F_k)]&=\sum_{i=1}^k\Phi^\prime_i(F_1,\ldots,F_k)F_i^\sharp\\
&=\lim\frac{1}{\theta}[(\Phi(F_1,\ldots,F_k))^\theta-\Phi(F_1,\ldots,F_k)].
\end{array}$$
\end{proof}
It follows that we may replace   $(\Phi(F_1,\ldots,F_k))^\theta$ by $\Phi(F_1^\theta,\ldots,F_k^\theta)$ in applying the method. 

Let us define an {\it equivalence relation} denoted $\cong$ in the set of functionals in   $L^2(\mathbb{P})$ depending on $\theta$ and differentiable in  $L^2$ at $\theta=0$ by 
$$F(\omega_1,\omega_2,\theta)\cong G(\omega_1,\omega_2,\theta)
\mbox{ if }\left(\frac{d}{d\theta}F|_{\theta=0}=\frac{d}{d\theta}G|_{\theta=0}\mbox{ and }F|_{\theta=0}=G|_{\theta=0}\right).$$

Let us also define a weaker equivalence relation denoted $\simeq$ for functionals in $L^0(\mathbb{P})$ depending on $\theta$ and differentiable in probability at $\theta=0$ by
$$F(\omega_1,\omega_2,\theta)\simeq G(\omega_1,\omega_2,\theta)
\mbox{ if }\left(\frac{d}{d\theta}F|_{\theta=0}=\frac{d}{d\theta}G|_{\theta=0}\mbox{ and }F|_{\theta=0}=G|_{\theta=0}\right)$$ the limits in the derivations being in probability.

\begin{Pro}\label{int2}If $H_t(\theta)\cong K_t(\theta)$ for all $t$ then
$$\int_0^tH_s(\theta)dY^\theta_s\cong\int_0^tG_s(\theta)dY^\theta_s.$$
\end{Pro}
\begin{proof}The equality of the value at zero of the two terms is evident, and differentiating the lefthand term in zero gives 
$$\int_0^t\frac{dH_s(\theta)}{d\theta}|_{\theta=0}dB_s+\int_0^tH_s(0)dM_s$$
which is equal to the derivative of the righthand term. \end{proof}
Let us consider a stochastic differential equation (SDE) with   $\mathcal{C}^1\cap Lip$ coefficients with respect to  the argument $x$
\begin{equation}\label{eds}X_t=x+\int_0^t\sigma(s,X_s)dB_s+\int_0^tb(s,X_s)ds.\end{equation}
Let us recall that we are considering chaotic extensions $F\mapsto F^\theta$ with respect to $Y^\theta=B\cos\theta+M\sin\theta$.\\
The following proposition shows that we can calculate the Malliavin gradient of the diffusion by perturbing the Brownian trajectories using an independent normal martingale such as a 
compensated Poisson process.
\begin{Pro}\label{derivEDS}The chaotic extension $X_t^\theta$ of the solution of {\rm(\ref{eds})} is equivalent (relation $\cong$) to the solution $Z_t(\theta)$ of the SDE
\begin{equation}\label{eds2}
Z_t(\theta)=x+\int_0^t\sigma(s,Z_{s}(\theta))dY^\theta_s+\int_0^tb(s,Z_s(\theta))ds.\end{equation}
\end{Pro}
\begin{proof} Let us denote $(X^n )_{n\in\N}$ (resp. $(Z^n )_{n\in\N}$) the approximating sequence in the  Picard iteration applied to equation satisfied by $X$ (resp. $Z$). 
We have first $X^0_t =Z^0_t =x$ and 
$$X_t^1 =x+\!\int_0^t\sigma(s,x)dB_s+\!\int_0^tb(s,x)ds.$$
 By Propositions  \ref{int} and \ref{comp}
$$(X_t^1)^\theta=\left(x+\!\int_0^t\sigma(s,x)dB_s+\!\int_0^tb(s,x)ds\right)^\theta\cong x+\!\int_0^t\sigma(s,x)dY^\theta_s+\!\int_0^tb(s,x)ds={Z^1_t(\theta)}$$
Then, 
 $$X_t^2=x+\!\int_0^t\sigma(s,X^1_s)dB_s+\!\int_0^tb(s,X^1_s)ds,$$
so that
$$\displaystyle(X_t^2)^\theta\displaystyle\cong x+\!\int_0^t\sigma(s,(X^1_s)^\theta)dY^\theta_s+\!\int_0^1b(s,(X^1_s)^\theta) ds$$
what gives by Propositions \ref{int2} and \ref{comp}:
$$ \displaystyle(X_t^2)^\theta\displaystyle\cong x+\!\int_0^t\sigma(s,Z^1_{s}(\theta))dY^\theta_s+\!\int_0^tb(s,Z^1_s(\theta))ds\\
=Z^2_t(\theta).
$$
By induction, we get easily that for all $n\in\N$, $(X^n_t )^\theta\cong Z^n_t (\theta)$. But we know that  $X^n_t$ converges to $X_t$ as $n$ tend to $+\infty$ not only in  $L^2$ but in $\mathbb{D}$ since the coefficients of the SDE are  Lipschitz (cf \cite{bouleau-hirsch2} Chap IV), and this implies that  $(X^n_t)^\theta$ converges to $X^\theta_t$ and that $\frac{d}{d\theta}(X^n_t)^\theta$ converges to $\frac{d}{d\theta}X_t^\theta$. 

Now,  $Z^n_t(\theta)$ converges to $Z_t(\theta)$. and its derivative  converges to a solution of  $$Z_t^\prime(\theta)=\int_0^t\sigma^\prime_x(s,Z_{s}(\theta))Z^\prime_s(\theta)dY^\theta_s+\int_0^t\sigma(s,Z_{s}(\theta))dY^{\theta+\pi/2}_s+\int_0^tb^\prime_x(s,Z_s(\theta))Z_s^\prime(\theta)ds$$ equation which has a unique  explicit solution  as linear equation  in $Z^\prime(\theta)$ which is the derivative of  $Z(\theta)$. That proves the  proposition. \end{proof}

\section{ The unit jump on the interval $[0,1]$.}

 In order to express the preceding results on $[0,1]$ with a single jump, we propose two different approaches.\\

\subsection{First approach}

We come back  to the case $M=\tilde{N}$ and  to express the preceding results on the interval   $[0,1]$,  we are conditioning  by $\{N_1=1\}$. This amounts to reasoning on  $\Omega_1\times\{N_1=1\}$ under the measure  $\mathbb{P}=\mathbb{P}_1\times\mathbb{P}_2$ which gives it the mass $e^{-1}$. Then the unique jump is uniformly distributed on  $[0,1]$. For a functional  $F\in L^2$ with expansion $F=\sum_n I_n(f_n)$, the expansion of the chaotic extension  $F^\theta=\sum_nI_n^\theta(f_n)$ considered on the event  $\{N_1=1\}$ is the same sum of stochastic integrals but with integrant the semi-martingale  $V_t(\theta)=B_t\cos\theta+(1_{\{t\geq U\}}-t)\sin\theta$, in other words $$F^\theta=\int_\mathcal{P}f(A)dV_A(\theta)\;\mbox{ on }\; \{N_1=1\}.$$

In the sequel, we use the fact that the notation 
$$\int_\mathcal{P}f(A)dS_A$$
makes sense for any semi-martingale $S$ which admits the decomposition
$$S_t =M_t + L_t,$$
where $M$ is a local martingale whose skew bracket is absolutely continuous w.r.t. the Lebesgue measure and $L$ an absolutely 
continuous process.\\
For example if $U$ is uniform on $[0,1]$, then 
$$1_{\{ t\geq U\}}= M_t -\log (1-t\wedge U),$$
with $M_t =1_{\{ t\geq U\}}+\log (1-t\wedge U)$ and $\langle M,M\rangle_t =-\log (1-t\wedge U)$.\\
By absolutely continuous change of probability measure we may remove the term in $-t\sin\theta$:
 \begin{Pro}\label{01}{\it Let $U$ be uniform on $[0,1]$ independent of $B$. We put  \break $R_t(\theta)=B_t\cos\theta+1_{\{t\geq U\}}\sin\theta$. Let be $F\in \mathbb{D}$, $F=\int_\mathcal{P}f(A)dB_A$.  we have
 $$\lim_{\theta\rightarrow0}\frac{1}{\theta}(\int_\mathcal{P}f(A)dR_A(\theta)-F)=D_UF\qquad\mbox{in probability}.$$
 }\end{Pro}
 \noindent For the proof let us state the 
 \begin{Le} Let be  $ \xi$ be an element in the Cameron-Martin space.\\
  a) If $F\in L^p$ for $p>2$ then  
 $$\lim_{t\rightarrow 0}\mathbb{E}[(F(B+t\xi)-F(B))^2]=0.$$
 b) If $F\in L^0$ then $F(B+t\xi)$ converges to $F$ in probability as $t$ tends to $0$.
 \end{Le}
\begin{proof}
 a) We develop the square. The first term is 
 $$\mathbb{E}\left[exp[t\int\dot{\xi}dB-\frac{t^2}{2}\|\dot{\xi}\|^2]F^2\right]$$
 as $F\in L^p$ $p>2$ it is uniformly integrable and it converges to $\mathbb{E}F^2$.
 
 For the rectangle term, it is easily seen by change of probability measure that
  $$\mathbb{E}[F(B+t\xi)G(B)]$$ converges to $\mathbb{E}[FG]$ for $G$ bounded and continuous.
 
 And we can reduce to this case by the above argument.
 
 b) We truncate  $F$. If $A_n=\{B : |F|\geq n\}$ by uniform int\'egrability we can find  $n$ such that the probability of  $A_n(B+t\xi)$ be $\leq \varepsilon$ for all $t$. The result comes now from part a).\end{proof}
 
 \noindent{\bf proof of  proposition \ref{01} :}
 
   Puting $C=\{N_1=1\}$, we are working under the probability measure  $\mathbb{Q}=e\times\mathbb{P}_1\times (\mathbb{P}_2|_C)$.

 The conditioning explained above yields the following relation in probability
\begin{equation}\label{cond}\lim_{\theta\rightarrow0}\frac{1}{\theta}(\int_\mathcal{P}f(A)dV_A(\theta)-F)=D_UF-\int_0^1D_sFds,\end{equation}
whose second member is   $\int_0^1D_sFd\tilde{N}_s$ restricted to   $\{N_1=1\}$. In order to have 
  $$\lim_{\theta\rightarrow0}\frac{1}{\theta}(\int_\mathcal{P}f(A)dR_A(\theta)-F)=D_UF\qquad\mbox{in probability},$$ we use that the identity map $j$ is a   Cameron-Martin function and that  $\int_0^1D_sFds=\langle DF,\frac{dj}{ds}\rangle_{L^2(ds)}$.

If we change of measure and take   $\exp(-B_1\sin\theta-\frac{\sin^2\theta}{2}).\mathbb{Q}$ relation (\ref{cond}) writes saying that, for all $\varepsilon>0$, 
$$\mathbb{Q}\left[\exp(-B_1\sin\theta-\frac{\sin^2\theta}{2})1_{C_\theta^\varepsilon}\right]$$
tends to zero, where we denote
$$C_\theta^\varepsilon=\{|\lim_{\theta\rightarrow0}\frac{1}{\theta}[\int_\mathcal{P}f(A)dR_A(\theta)-F(B+j\sin\theta)]\hspace{2cm}\qquad\qquad$$
$$\qquad\hspace{2cm}\qquad-D_UF(B+j\sin\theta)-\int_0^1D_sF(B+j\sin\theta)ds|\geq\varepsilon\}.$$

1) Let us observe that  
$$\mathbb{Q}\left[(\exp(-B_1\sin\theta-\frac{\sin^2\theta}{2})-1)1_{C_\theta^\varepsilon}\right]$$
tends to zero. What reduces to study  
$\mathbb{Q}[C_\theta^\varepsilon].$

2) We know that under  $\mathbb{Q}$,
$\frac{1}{\theta}[F(B+j\sin\theta)-F(B)]$ converges in probability to $\int_0^1D_sFds$. The other two terms are processed by the lemma. 

We obtain indeed that under $\mathbb{Q}$, $\frac{1}{\theta}[\int_\mathcal{P}f(A)d(B\cos\theta+1_{\{.\geq U\}}\sin\theta)_A-F]$ converges in probability to $D_UF$.
 \hfill$\Box$\\
 
If we are conditioning by the event $\{N_1=1\}$ the result of  Proposition \ref{derivEDS}, the equation satisfied by  $Z(\theta)$ may be written
$$\label{eds2}\hspace{-1cm}
Z_t(\theta)\displaystyle=x+\int_0^t\sigma(s,Z_{s}(\theta))d(B_s\cos\theta+(1_{\{s\geq U\}}-s)\sin\theta)+\int_0^tb(s,Z_s(\theta))ds$$
As in the proof of proposition   \ref{01} an absolutely continuous change of probability measure allows to remove the term in  $-s\sin\theta$ if we replace the result  $D_UZ_t-\int_0^tD_sZ_tds$ by $D_UZ_t$.

This change being done, the value at  $\theta=0$ and the derivative at $\theta=0$ of $Z(\theta)$ are the same as those of  $\eta(\theta)$ solution of the SDE
$$\eta_t(\theta)=x+\int_0^t\sigma(s,\eta_{s}(\theta))(dB_s+1_{\{s\geq U\}}\theta)+\int_0^tb(s,\eta_s(\theta))ds.$$ In other words we obtain the following result which might have been easily directly verified
\begin{Pro}\label{fppbrownien} The gradient of Malliavin $D_uX_t^x$ of the solution of the SDE
$$X_t^x=x+\int_0^t\sigma(s,X_s^x)dB_s+\int_0^tb(s,X_s^x)ds$$ may be computed by considering the solution of the equation $$X_t^x(\theta)=x+\int_0^t\sigma(s,X_s^x(\theta))d(B_s+\theta1_{\{s\geq u\}})+\int_0^tb(s,X_s^x(\theta))ds$$
and taking the derivative in  $\theta$ at $\theta=0$.
\end{Pro} Let us remark that since $u$ is defined $du$-almost surely, we may in the equation defining   $X_t^x(\theta)$ put either $\sigma(s,X_s^x(\theta))$ or $\sigma(s,X_{s-}^x(\theta))$.

Let us perform the calculation suggested in the proposition. That gives for $u<t$
$$X_t^x(\theta) =x + \int_0^u \sigma(s,X_s^x)dB_s +\theta\sigma(u,X_u^x)+\int_u^t\sigma(s,X_s^x(\theta)dB_s
+\int_0^ub(s,X_s^x(\theta))ds
$$ $$+\int_u^tb(s,X_s^x(\theta))ds$$
\begin{equation}\label{sept}X_t^x(\theta)=X_u^x+\theta\sigma(u,X_u^x)+\int_u^t\sigma(s,X_s^x(\theta))dB_s+\int_u^tb(s,X_s^x(\theta))ds.\end{equation}
Hence, denoting $Y^x_t=1+\int_0^t\sigma^\prime_X(s,X^x_s)Y^x_sdB_s+\int_0^tb^\prime_X(s,X^x_s)Y^x_sds\; $ and
\begin{equation}\label{huit}X_{u,t}^y=y+\int_u^t\sigma(s,X_{u,s}^y)dB_s+\int_u^tb(s,X_{u,s}^y)ds\end{equation} we see by comparing (\ref{sept}) and  (\ref{huit}) that 
\begin{equation}\label{compo}X_t^x(\theta)=X_{u,t}^{(X_u^x+\theta\sigma(u,X_u^x))}\quad\mbox{et}\quad X_t^x=X_{u,t}^{X_u^x} \end{equation}
 so, derivating (\ref{sept}) with respect to $\theta$ and (\ref{huit}) with respect to $y$, and then derivating the second relation of (\ref{compo}) with respect to $x$
\begin{equation}D_uX_t^x=\sigma(u,X_u^x)\left[\frac{dX_{u,t}^y}{dy}|_{y=X_u^x}\right]=\sigma(u,X_u^x)\frac{Y^x_t}{Y^x_u}.
\end{equation}
This is a fast way of obtaining this classical result  ({transfert principle by the flow} of Malliavin cf \cite{malliavin} Chap VIII). 

Proposition \ref{fppbrownien} is the lent particle formula for the Brownian motion. 

We see that the method of proof allows to obtain this formula without  sinus nor cosinus for general  $F$ in $\mathbb{D}$ provided that we be able to find a functional regular in  $\theta$ equivalent to the chaotic extension of  $F$. In particular the example of the introduction generalises in the following way : if 
$$F=\Phi\left(\int_{s_1<\cdots<s_{k_1}}f_{k_1}dB_{s_1}\cdots dB_{s_{k_1}},\ldots,
\int_{s_1<\cdots<s_{k_n}}f_{k_n}dB_{s_1}\cdots dB_{s_{k_n}}\right)$$
with $f_{k_i}\in L^2(\lambda_{k_i})$ and $\Phi$ of class $\mathcal{C}^1\cap Lip$, we have
$$\lim\frac{1}{\theta}\left[\Phi\left(\int_{s_1<\cdots<s_{k_1}}f_{k_1}d(B_{s_1}+\theta1_{\{s\geq u})\cdots d(B_{s_{k_1}}+\theta1_{\{s\geq u}),\ldots\right)-F\right]=D_uF.
$$ The limit is in probability as in Proposition \ref{01}.\\

{\subsection{Second approach }
 
  Instead of performing a change of probability measure to remove the term $-t\sin\theta$ as in the previous approach, we consider $M$ a L\'evy process with L\'evy measure $\frac12 (\delta_{-1}(dx)+\delta_1 (dx))$ in place of $\tilde{N}$. Let us remark that we might have considered any Lévy process whose Lévy measure has mean $0$ and variance $1$. $M$ can be expressed as 
  $$\forall t\geq 0 ,\ M_t =\sum_{n=1}^{N_t} J_n,$$
  where $N$ is a Poisson process with intensity $1$ and $(J_n )_n$ a sequence of i.i.d. variables, independent of $N$ such that 
  $$\mathbb{P}_2 (J_1 =1)=\mathbb{P}_2 (J_1 =-1)=\frac12 .$$ 
  \begin{Pro}\label{Cond}{\it Let $U$ be uniform on $[0,1]$ independent of $B$. We put  \break $R_t(\theta)=B_t\cos\theta+1_{\{t\geq U\}}\sin\theta$. Let be $F\in \mathbb{D}$, $F=\int_\mathcal{P}f(A)dB_A$.  we have
 $$\lim_{\theta\rightarrow0}\frac{1}{\theta}(\int_\mathcal{P}f(A)dR_A(\theta)-F)=D_UF\qquad\mbox{in}\  L^2 (\bbP ).$$
 }\end{Pro} 
  \begin{proof} We denote by $U_1$ the time of the first jump and by $\widetilde{\bbP}_2$ the conditional law $e{\bf 1}_{\{ N_1 =1\}} \bbP_2$. \\
We consider
  $${R}^1_t(\theta)=B_t\cos\theta+ 1_{\{t\geq U_1\}}\sin\theta\makebox{ and }\widetilde{R}^1_t(\theta)=B_t\cos\theta+J_1 1_{\{t\geq U_1\}}\sin\theta.$$
 The chaotic extension related to $Y_t (\theta)=B_t\cos\theta + M_t \sin\theta$ satisfies 
  $$ \forall \theta,\ F^\theta =\int_\mathcal{P}f(A)dY_A(\theta)=\int_\mathcal{P}f(A)d\widetilde{R}^1_A(\theta)\ \bbP_1 \times\widetilde{\bbP}_2 \makebox{ a.e.}$$
As a consequence of Proposition \ref{derivee2}, we have in the sense of $L^2 ( \bbP_1 \times\widetilde{\bbP}_2)$:
$$\lim_{\theta\rightarrow0}\frac{1}{\theta}(F^\theta-F)=\int D_s F \, dM_S =J_1 D_{U_1}F.$$
Then, we remark that
$$\int_\mathcal{P}f(A)d{R}^1_A(\theta) = J_1 F^\theta + (1-J_1) \cos\theta \int_\mathcal{P}f(A)dB_A,$$
and obtain
 $$\lim_{\theta\rightarrow0}\frac{1}{\theta}(\int_\mathcal{P}f(A)dR^1_A(\theta)-F)=D_{U_1}F\qquad\mbox{in}\  L^2 (\bbP_1 \times \widetilde{\bbP}_2 ).$$

  \end{proof}

\subsection{Another example}}

\noindent{\bf Remark 5}. When we enlarge the field of validity of the calculus, by using the equivalence relation  $\simeq$ instead of $\cong$ to functionals in  $L^0(\mathbb{P})$ depending on $\theta$ and differentiable  in probabibility at $\theta=0$, the authorized functional calculus becomes $C^1$ instead of $C^1\cap Lip$.

For  example let us consider a c\`adl\`ag  process $K$ independent of $B$.
And let us put $M(B)=\sup_{s\leq 1}(B_s+K_s)$. 
\begin{Pro} $M^\theta\cong M(B+\theta1_{\{.\geq U\}})\cong M(B\cos\theta+1_{\{.\geq U\}}\sin\theta)$.
\end{Pro}
\begin{proof} For the first equivalence, the value at  $\theta=0$ is of course correct, and about the derivative we have  
$$\begin{array}{rl}
M(B\;+&\!\!\!\theta1_{\{.\geq U\}})\\
&=\sup_{s\leq 1}((B_s+K_s)1_{\{s<U\}}+(B_s+\theta+K_s)1_{\{s\geq U\}})\\
&=\max(\sup_{s<U}(B_s+K_s),\sup_{s\geq U}(B_s+\theta+K_s))\\
\end{array}
$$
 Thus we have the convergence in probability
$$\lim_{\theta\rightarrow0}\frac{1}{\theta}(M(B+\theta1_{\{.\geq U\}})-M(B))=1_{\{\sup_{s\geq U}(B_s+K_s)\geq\sup_{s<U}(B_s+K_s)\}}
$$
This result is correct  cf \cite{nualart-vives2} for the  gradient of $M$. The second equivalence is similar what shows the proposition. This implies that  $M$ possesses a density since $\Gamma[M]=0$ cannot be true except if  $\sup(B_s+K_s)<K_0$ what is impossible since  $K$ is independent of $B$.
\end{proof}

%


\begin{thebibliography}{00}

\bibitem{akr}{\sc Albeverio S., Kondratiev Y.} and {\sc R\"ockner M.} "Analysis and geometry on configuration spaces" {\it J. Funct. Analysis} 154, 444-500, (1998).
\bibitem{bouleau-denis} {\sc Bouleau N.} and {\sc Denis L.} ``Energy image density property and the lent particle method for Poisson measures" {\it Jour. of Functional Analysis} 257 (2009) 1144-1174. available online: http//dx.doi.org/10.1016/j.jfa.2009.03.004
\bibitem{bouleau-denis2} {\sc Bouleau N.} and {\sc Denis L.} ``Application of the lent particle method to Poisson driven SDE's", {\it Probability Theory and Related Fields} 151, 403-433, (2011).
\bibitem{bouleau-hirsch2}{\sc Bouleau N.} and {\sc Hirsch F.} {\it Dirichlet Forms and Analysis on Wiener Space} De Gruyter (1991).
\bibitem{cont}{\sc Cont R.} and {\sc Fournie D. } "Functional It\^o calculus and stochastic integral representation of martingales", {\it arXiv:1002.2446 }(2010).
\bibitem{dellacherie}{\sc Dellacherie C., Maisonneuve B.} and {\sc Meyer P.-A.} {\it Probabilit\'es et Potentiel} Chap XVII \`a XXIV, Hermann 1992.
\bibitem{dermoune}{\sc Dermoune A.} "Distributions sur l'espace de P. L\'evy et calcul stochastique" {\it Annales de l'IHP}, B, 26, n1, (1990), 101-119.
\bibitem{dupire}{\sc Dupire B.} "Functional It\^o calculus", {\it papers.ssrn.com} (2009).
\bibitem{oksendal}{\sc Di Nunno, \O ksendal B., Proske F.} {\it Malliavin Calculus for L\'evy Processes with applications to Finance} Springer 2009.
\bibitem{ma-rockner2}{\sc Ma} and {\sc R\"ockner M.} ``Construction of diffusion on configuration spaces" {\it Osaka J. Math.} 37, 273-314,  (2000).
\bibitem{malliavin}{\sc Malliavin P.}{\it Stochastic Analysis} Springer 1997.
\bibitem{nualart-vives}{\sc Nualart D.} and {\sc Vives J.} ``Anticipative calculus for the Poisson process based on the Fock space", {\it S\'em. Prob. XXIV}, Lect. Notes in M. 1426, Springer (1990).
\bibitem{nualart-vives2}{\sc Nualart D.} and {\sc Vives J.} ``Continuit\'e absolue de la loi du maximum d'un processus continu" {\it C. R. Acad. Sci.} sI, 307, 349-354, (1988)
\bibitem{privault}{\sc Privault N.} {\it Stochastic Analysis in Discrete and Continuous Setting,} Springer 2009.
\bibitem{saphar}{\sc Saphar, P.} "Fonctions de Bessel" {\it Encyclop\ae dia Universalis} 1997.
\bibitem{surgailis}{\sc Surgailis D.} `On multiple Poisson stochastic integrals and associated Markov processes" {\it Probability and Mathematical Statistics} 3, 2, 217-239, (1984)
\bibitem{wu}{\sc Wu L.} "Construction de l'op\'erateur de Malliavin sur l'espace de Poisson" {\it S\'em. Probabilit\'e XXI} Lect. Notes in M. 1247, Springer (1987).
\end{thebibliography}
\end{document}